\documentclass[11pt,reqno,oneside]{amsart}
 \usepackage{amsmath,amssymb}%,cite,epsfig,xspace,alltt,verbatim,bm}
\columnwidth\textwidth

\theoremstyle{plain}

\newtheorem{proposition}{Proposition}[section]
\newtheorem{corollary}{Corollary}[section]
\theoremstyle{definition}

\theoremstyle{remark}

\newcommand{\Real}{\mathbb R}

\newcommand{\Int}{\mathbb Z}
\newcommand{\Intpos}{\Int^{> 0}}

\newcommand{\eps}{\varepsilon}

\newcommand{\A}{{\mathcal A}}
\newcommand{\C}{{\mathcal C}}
\newcommand{\I}{{\mathcal I}_\eps}
\newcommand{\Ahat}{{\widehat {\mathcal A}}_\eps}
\newcommand{\Chat}{{\widehat {\mathcal C}}_\eps}

\newcommand{\E}{{\mathcal E}}

\def\XXint#1#2#3{{\setbox0=\hbox{$#1{#2#3}{\int}$}
     \vcenter{\hbox{$#2#3$}}\kern-.5\wd0}}

\newcommand{\lpnorm}[2]{\left\|#1\right\|_{\ell^{#2}_\eps}}

\title{There is no pointwise consistent quasicontinuum energy}
\author{Matthew Dobson}
\address{CERMICS - ENPC,
6 et 8 avenue Blaise Pascal,
Cité Descartes - Champs sur Marne,
77455 Marne la Vallée Cedex 2 (FRANCE)}
\begin{document}
\begin{abstract}
Much work has gone into the construction 
of quasicontinuum energies that reduce the coupling error
along the interface between atomistic and continuum regions.
The largest consistency errors are typically 
pointwise $O(\frac{1}{\eps})$ errors, and in some cases this has been
reduced to pointwise $O(1)$ errors.
In this paper we show that one cannot create a coupling method
using a finite-range coupling interface that has o(1)-consistency
in the interface, and we use this to give an upper bound on the order
of convergence in discrete $w^{1,p}$-norms in 1D.
\end{abstract}
\maketitle
\section{Introduction}
Atomistic to continuum coupling schemes attempt to provide the high
accuracy of an atomistic simulation while also providing the
computational savings of a coarse-grained continuum approximation.  
Such simulations are used in the study of localized defects that 
interact with a long-range elastic field, such as crack tips, voids,
and interstitials.  The quasicontinuum method directly couples an atomistic region, where
each atom interacts with every other atom within a cut-off radius, to a
continuum region, whose strain energy density is computed by assuming
Cauchy-Born kinematics (locally uniform deformation gradients).  Many
variants of the scheme exist that differ on how to couple the regions
together.

One can classify quasicontinuum methods based on whether they create
a total energy or directly couple
forces in such a way that there does not exist a total energy.
The original formulation of the quasicontinuum
method~\cite{tadm96,mill02} creates a total energy by defining an energy
for each degree of freedom (atom or finite element)
and using a weighted
sum.  Unfortunately, the
model contains spurious forces at the atomistic to continuum interface
called ghost forces~\cite{rodn99}.
These forces are a pointwise $O(\frac{1}{\eps})$ consistency error in
the linearized scheme and reduce the order of convergence with respect to the interatomic spacing parameter $\eps$~\cite{dobs09b}.  
Subsequent methods such
as the quasinonlocal quasicontinuum method~\cite{shim04} or the
geometrical reconstruction scheme~\cite{e06} removed some ghost
forces, but were restricted on the number of neighbors or to specific
interface geometry in 2D and 3D.
Recently, the ECC and ACC quasicontinuum variants have been proposed in 1D
and 2D , and these schemes 
are free of ghost forces for the case of pair-potential
interactions~\cite{shap_consistent}.  A related
method, derived as an extension of QNL, is formulated in 1D
in~\cite{li_gen_qnl}. 

An alternative approach to consistency is the use of the force-based
quasicontinuum method~\cite{curt03,dobs08,makr_qcfs}, which assigns forces to
any degree of freedom, atomistic or continuum, as though the entire
model were of the same type.
Since the continuum model
and atomistic model are $o(1)$-consistent and there are no special
force laws in the atomistic-to-continuum interface, this creates a
globally point-wise consistent scheme for any geometry.  However, this method
is more difficult to analyze since it does not derive from an
energy~\cite{dobs10,dobs_spec}.
Solutions to the force-based equations can be approximated by employing
the ghost-force correction
scheme~\cite{shen99,dobs08}.

Here we ask the question of whether it is possible to create a
quasicontinuum energy which is globally $o(1)$-pointwise consistent.
We will find in Section~\ref{noconsistent} that
even in the case of a linear, 1-D model it is impossible to make an $o(1)$
consistent blending of the atomistic and continuum region for any finite-sized
interface region.  While such a finiteness assumption is common, a quasicontinuum force coupling with growing interfaces is treated in~\cite{lu_qcf}.
In Section~\ref{sec:error}, we will show how this leads to a bound on the maximal
order of convergence in discrete  $w^{1,p}$-norms in 1D.  Convergence analysis
for related quasicontinuum schemes include~\cite{lin05,gunz10,abdulle_multi,kote_blend}.

\section{Models}

We denote the positions of a periodic atomistic chain by
\begin{equation*}
y \in \Real^N = \{ \Real^\Int \ :
      \ y_{\ell+N} = y_\ell+F \text{ for all } \ell \in \Int \},
\end{equation*}
where $F$ denotes the macroscopic deformation gradient for the
periodic length and $N \in \Int_{\geq 0}$ denotes the number of 
atoms in the periodic cell.  We scale the reference length of the 
the interatomic spacing by  $\eps = \frac{1}{N}$ in order to have
a well-defined continuum limit.
We denote the displacements from the uniformly
deformed position by $u_i = y_i - F \eps i$ for all $i\in \Int.$
We define the backward difference quotient
\begin{equation*}
D_r u_i = \frac{u_{i} - u_{i-r}}{r \eps} \qquad r \in \Intpos,
\end{equation*}
where we will write $D u_i$ to denote $D_1 u_i.$  We will write the 
vector of such differences as $D u.$  We likewise define the
centered second-difference quotient
\begin{equation*}
D_r^2 u_i 
= \frac{u_{i+r} - 2 u_i + u_{i-r}}{r^2 \eps^2}  \qquad r \in \Intpos.
\end{equation*}

\subsection{Atomistic and continuum models}
We consider the atomistic energy composed of pairwise interactions with the
nearest $R$ neighbors, given by
\begin{equation}
\label{ea}
\E_\eps^a(u) = \sum_{r=1}^R \sum_{i=1}^N \eps \phi(r F +r D_r u_i),
\end{equation}
where $\phi$ denotes the interatomic interaction potential and $R$
denotes the finite-range cut-off.  The energy is scaled by the
lattice spacing $\eps$ so that there is a finite energy in the limit
$N\rightarrow \infty.$   The continuum approximation uses the same pair
potential and approximates the total energy by
\begin{equation}
\label{ec}
\E^c_\eps(u) = \sum_{r=1}^R \sum_{i=1}^N \eps \phi(r F + r D u_i).
\end{equation}
That is, interactions of the form $\phi\Big(\frac{y_i -
y_{i-r}}{\eps}\Big)$ are replaced by  $\phi\Big(r \frac{y_i -
y_{i-1}}{\eps}\Big)$.  The continuum energy is an accurate
approximation when $y_i$ is smooth.
In practical applications, the cost of computing the continuum energy is
reduced by computing the energy as a sum over the nodes in a piecewise
linear mesh rather than a sum
over all atoms.

Expanding the pairwise 
interactions, we have
\begin{equation*}
\phi(r F + D_r u_i) \approx
\phi(r F) + \phi'(r F) D_r u_i + \frac{1}{2} \phi''(r F) (D_r u_i)^2 + \cdots
\qquad \text{ for } r \in \Intpos.
\end{equation*}
For the atomistic model, the (scaled) forces on the atoms are
\begin{equation}
\label{atom_force}
\begin{split}
(L_\eps^a u)_i &= - \frac{1}{\eps} \frac{\partial \E^a(u)}{\partial u_i}
          = - \sum_{r=1}^R r^2 \phi''(rF) D_r^2 u_i,
\end{split}
\end{equation}
and for the continuum model, the (scaled) forces on the nodes are
\begin{equation}
\label{cont_force}
\begin{split}
(L_\eps^c u)_i &= - \frac{1}{\eps} \frac{\partial \E^c(u)}{\partial u_i}
          = - \sum_{r=1}^R r^2 \phi''(rF) D^2 u_i.
\end{split}
\end{equation}

\subsection{Quasicontinuum coupling}

A quasicontinuum energy couples the atomistic and continuum energies, by
partitioning the chain into an atomistic region $\A$ and a continuum region
$\C$ that satisfy
\begin{equation}
\label{partition}
\A \cup \C = (0,1] \quad \text{ and } \quad \A \cap \C = \emptyset.
\end{equation}
We assume that $\A$ is a finite union of intervals.  An atom $i$ is in the
atomistic region if $\eps i \in \A.$ Away from the boundary of these
regions, the energy is computed using the energy contributions
$\E^a_i$ or $\E^c_i.$  Near the interface between the regions, there
can exist special interfacial energies that we assume are finite range
and which do not change in form with $\eps.$  Note that the interfacial
interactions may be longer-ranged than the atomistic ones.  

The above assumptions imply that away from the atomistic to continuum 
interface and for sufficiently large $N,$ 
the atoms in the continuum region only feel continuum forces~\eqref{cont_force} and
the atoms in the atomistic region only feel atomistic forces~\eqref{atom_force}.
We define $\Ahat$ and $\Chat$ as the interior portions of the atomistic and continuum region that only feel interactions from within their specific region.  We define the interface region $\I = [0,1] \setminus (\Ahat \cup \Chat).$  In the following,
we will focus on a single interface between the atomistic and continuum regions.

\section{Consistency at the Interface}
\label{noconsistent}
We say that a family of operators $L_\eps$ is
$o(1)$-consistent with the atomistic operator
$L^a$ if
\begin{equation}
\label{cons1}
\lim_{\eps \rightarrow 0} \lpnorm{L_\eps u - L_\eps^a u}{\infty} = 0
\end{equation}
for every $u \in C_{per}^2(\overline{\Omega}),$ where for every $N,$ $u \in C_{per}^2(\overline{\Omega})$ defines a vector $u \in \Real^N$ by $u_i = u(i/N).$
Equivalently, we have the local consistency requirement that the equations
are consistent if
\begin{equation}
\label{cons}
(L_\eps u - L_\eps^a u) = 0 \qquad \text{ for the three vectors } u_j=1, j, j^2.
\end{equation}
We note that the continuum operator $L_\eps^c$ is consistent with the 
atomistic operator.
In the following, we show that one cannot satisfy these consistency
equations for a quasicontinuum energy
even in
the simplified case of second-neighbor interactions in a
1-D chain with harmonic pair potentials.

We now restrict ourselves to the second-neighbor ($R=2$) case for the
atomistic and continuum energies.  We write the quasicontinuum operator
in terms of its first and
second-neighbor contributions
\begin{equation}
\label{orders}
\begin{split}
L_\eps^{qc} &= \frac{1}{\eps^2} ( \phi''(F) L^{qc}_1 + \phi''(2F) L^{qc}_2),
\end{split}
\end{equation}
where the operators $L^{qc}_1$ and $L^{qc}_2$ are independent
of $\phi$ and $\eps$.  This form is possible when we assume that the corresponding energy is a function of the ``strains'' $D_r u$ only.  We then consider possible choices for quasicontinuum 
energies.  Since the first neighbor term for both atomistic and
continuum operator are identical, it is standard to
take the same operator for the first neighbor term of a quasicontinuum energy.
This gives $(L^{qc}_1 u)_j = - \eps^2 D^2 u_j.$  The second neighbor term is a symmetric operator satisfying
\begin{equation}
\label{blend}
(L^{qc}_2 u)_i = \begin{cases}
\displaystyle - u_{j+2} + 2 u_j - u_{j-2} = - 4 \eps^2 D_2^2 u_i %\frac{- u_{j+2} + 2 u_j - u_{j-2}}{\eps^2},
      & \eps i \in \Ahat,  \\[6pt]
\displaystyle - 4u_{j+1} + 8 u_j - 4u_{j-1} = - 4 \eps^2 D^2 u_i %\frac{- 4u_{j+1} + 8 u_j - 4u_{j-1}}{\eps^2},
      & \eps i \in \Chat, \\[6pt]
\end{cases}
\end{equation}
where we have not yet specified the interaction law in the interface $\I.$
We now consider one boundary between atomistic and continuum regions.
Let $m$ denote the number of atoms in the interval of the atomistic to continuum
interface $\I$ surrounding
the boundary between the regions.  This is represented in the operator
$L^{qc}_2$ by an $m \times m$ set of coefficients, $(L^{qc}_2)_{ij}$ for
$1 \leq i,j \leq m.$
We have labeled the interface starting at $j=1$ and will assume without loss
of generality that $(L_2^{qc} u)_j$ is continuum for $j<1$
and atomistic for $j > m.$

\begin{proposition}
There is no symmetric linear quasicontinuum operator $L_\eps^{qc}$ 
satisfying~\eqref{cons}.
\end{proposition}

\begin{proof}
We suppose toward contradiction that $L^{qc}$ satisfies~\eqref{cons}.  
Then, for all $i$, we have that 
\begin{equation}
\label{cons2}
\sum_{j} (L_{ij}^{qc} - L_{ij}^a) j = 0  \qquad \text{and} \qquad
\sum_{j} (L_{ij}^{qc} - L_{ij}^a) j^2 = 0,
\end{equation}
where $j$ runs over all indices for our periodic chain, though due to the limited range
of interactions outside the interface, we may restrict $j=-1,\dots,m+2.$
Now, we will sum the equations above for $i=1,\dots,m,$ pre-multiplying by $i^2$ and
$i,$ respectively, in order to cancel
out the interface variables in the interior.  We note that by
symmetry, $L_{ij}^{qc} = L_{ij}^c$ for $j < 1$ and $L_{ij}^{qc} = L_{ij}^a$ for $j > m.$ 
We have
\begin{equation*}
\begin{split}
0 &= \sum_{i=1}^m \sum_{j=-1}^{m+2} i^2 (L_{ij}^{qc} - L_{ij}^a) j -
\sum_{i=1}^m \sum_{j=-1}^{m+2} i (L_{ij}^{qc} - L_{ij}^a) j^2, \\
&= \sum_{i=1}^m \sum_{j=1}^{m} (L_{ij}^{qc} - L_{ij}^a) (i^2 j - i j^2)
+ \sum_{i=1}^m \sum_{j\in\{-1,0\}} (L_{ij}^{c} - L_{ij}^a) (i^2 j - i j^2)\\
&\qquad+ \sum_{i=1}^m \sum_{j\in\{m+1,m+2\}} (L_{ij}^{a} - L_{ij}^a) (i^2 j - i j^2)
\end{split}
\end{equation*}
We cancel the first summation by the symmetry of both $L^{qc}_\eps$ and $L^a,$ and the third
summation cancels identically.  The second summation has a single non-zero term
$i=1, j=-1,$ so that we have
\begin{equation*}
\begin{split}
0 &= \sum_{i=1}^m \sum_{j\in\{-1,0\}} (L_{ij}^{c} - L_{ij}^a) (i^2 j - i j^2)\\
&= 1 (-1 -1)= -2.
\end{split}
\end{equation*}
We have arrived at a contradiction to our assumption that $L^{qc}$ was  consistent~\eqref{cons2}.
\end{proof}

\section{Global convergence error}
\label{sec:error}
The result above shows that any linear quasicontinuum scheme must
have at least O(1) consistency error in the $\ell^\infty$ norm.  We can 
use the above bound on interfacial consistency error to give
an upper bound
on the order of convergence for any quasicontinuum energy.
We recall the definitions of the discrete $\ell_\eps^p$-norms on $\Real^N,$
\begin{equation*}
\begin{split}
\lpnorm{u}{p} &:= \left(\eps \sum^{N}_{j=1}  |u_j|^p\right)^{1/p},
	\qquad 1\le p<\infty,\\
\lpnorm{u}{\infty} &:= \max_{1\leq j\leq N}  |u_j|.
\end{split}
\end{equation*}
Provided that $L_\eps^{qc}$
is shift-invariant, that is, it satisfies~\eqref{cons} for $v_j = 1,$ we can
write the operator in terms of differences,
$$L_\eps^{qc} u= \widetilde{L}_\eps^{qc} D u,$$
where we note that the coefficients of $\eps \widetilde{L}_\eps^{qc}$
are bounded independently from $\eps.$
We have the following lower bound for the error.
\begin{corollary}
There exists $u \in C_{per}^2(\overline{\Omega})$
such that for $L_\eps^{qc} u_{qc} = L_\eps^a u$ the error $e = u - u_{qc}$ satisfies
\begin{equation*}
\begin{split}
\lpnorm{D e}{p} \geq C \eps^{1+1/p}
\end{split}
\end{equation*}
for $1 \leq p \leq \infty.$
\end{corollary}

\begin{proof}
Since the coefficients of $\eps \widetilde{L}_\eps^{qc}$ are bounded, we
have the inequality $\lpnorm{L^{qc} v}{\infty} \leq \frac{C}{\eps}
\lpnorm{D v}{\infty}.$  Using the fact that the operator is not o(1)-consistent,
we choose $u \in C_{per}^2$ such that $$\lim_{\eps \rightarrow 0} \lpnorm{(L_{\eps}^{a} - L_{\eps}^{qc}) u}{\infty} \geq c.$$
We then write
\begin{equation*}
\begin{split}
c &\leq \lpnorm{ (L_{\eps}^{a} - L_{\eps}^{qc}) u}{\infty}
=\lpnorm{ L_{\eps}^{qc} e}{\infty}
\leq \frac{C}{\eps} \lpnorm{D e}{\infty}.
\end{split}
\end{equation*}
Which implies $\lpnorm{D e}{\infty} \geq C \eps.$  The estimate for other
norms follows from the equivalence estimates for $\ell^p$ norms,
$\lpnorm{D e}{p} \geq \eps^{1/p} \lpnorm{D e}{\infty} \geq C \eps^{1+1/p}.$
\end{proof}

This result implies that the rates of convergence obtained in 1D for the
quasinonlocal quasicontinuum method in~\cite{dobs09b, ming09} are asymptotically
optimal among all possible quasicontinuum energies with finite-range interface
interactions.

\section{Conclusion}

There is much past and ongoing work to create consistent atomistic-to-continuum
coupling schemes, and a major obstacle are the O($\frac{1}{\eps}$)-ghost forces
common in such models.  Quasicontinuum energies that are O(1)-consistent
have been constructed in 1D and 2D, at least for the case of pair-potential
interactions.
In this work we have shown the impossibility of improving the asymptotic
truncation error for a linear quasicontinuum energy in 1D.  Since a higher
dimensional problems can behave like a 1D configuration due to the geometry (such as simple
uniaxial
strain applied parallel to a planar atomistic-to-continuum interface),
this implies that
it is impossible to make a general o(1)-consistent coupling energy in higher
dimensions as well.  The effect on the rate of convergence shows that
in 1D the quasinonlocal quasicontinuum energy and its generalizations are asymptotically optimal
among quasicontinuum energies with
respect to $w^{1,p}$-norm convergence.  However, for other quantities of interest,
such as the critical strain for dislocation movement~\cite{kote11}, such a method
may not be asymptotically optimal.

\section*{Acknowledgements}
The author wishes to thank Mitchell Luskin for many helpful comments on the manuscript.
This work is supported in part by the NSF Mathematical Sciences Postdoc Research Fellowship and ANR PARMAT (ANR-06-CIS6-006).

\end{document}